\documentclass[12pt]{article}
\usepackage{amssymb}
\usepackage{amsthm}
\usepackage{amsmath}
\usepackage{graphicx}
\usepackage{enumerate}

\DeclareMathOperator{\Con}{Con}

\DeclareMathOperator{\Tol}{Tol}

\newtheorem{theorem}{Theorem}

\newtheorem{lemma}[theorem]{Lemma}
\newtheorem{proposition}[theorem]{Proposition}
\newtheorem{remark}[theorem]{Remark}
\newtheorem{example}[theorem]{Example}
\newtheorem{corollary}[theorem]{Corollary}
\title{Tolerances on posets}
\author{Ivan~Chajda and Helmut~L\"anger}
\date{}
\begin{document}

\footnotetext{Support of the research of both authors by the Austrian Science Fund (FWF), project I~4579-N, and the Czech Science Foundation (GA\v CR), project 20-09869L, entitled ``The many facets of orthomodularity'', as well as by \"OAD, project CZ~02/2019, entitled ``Function algebras and ordered structures related to logic and data fusion'', and, concerning the first author, by IGA, project P\v rF~2021~030, is gratefully acknowledged.}

\maketitle

\begin{abstract}
The concept of a tolerance relation, shortly called tolerance, was studied on various algebras since the seventieth of the twentieth century by B.~Zelinka and the first author (see e.g.\ \cite{CZ} and the monograph \cite{Ch} and the references therein). Since tolerances need not be transitive, their blocks may overlap and hence in general the set of all blocks of a tolerance cannot be converted into a quotient algebra in the same way as in the case of congruences. However, G.~Cz\'edli (\cite{CZ}) showed that lattices can be factorized by means of tolerances in a natural way, and J.~Grygiel and S.~Radelecki (\cite{GR}) proved some variant of an Isomorphism Theorem for tolerances on lattices. The aim of the present paper is to extend the concept of a tolerance on a lattice to posets in such a way that results similar to those obtained for tolerances on lattices can be derived.
\end{abstract}

{\bf AMS Subject Classification:} 08A02, 08A05, 06A06, 06A11

{\bf Keywords:} Poset, tolerance relation, congruence on a poset, block, directed, convex, relatively complemented poset, quotient poset by a tolerance, Isomorphism Theorem

Tolerance relation play an important role both in algebra and in applications. For corresponding results, examples and references see the monograph \cite{Ch}. Tolerances were studied by several authors within the last decades. Up to now, tolerances were treated on various algebras, in particular on lattices, but not on posets. The aim of this paper is to extend this investigation also to posets.

Let $\mathbf L=(L,\vee,\wedge)$ be a lattice. A {\em tolerance} on $\mathbf L$ is a reflexive and symmetric binary relation $T$ on $L$ satisfying the following condition:
\begin{itemize}
\item If $(x,y),(z,u)\in T$ then $(x\vee z,y\vee u),(x\wedge z,y\wedge u)\in T$.
\end{itemize}
The congruences on $\mathbf L$ are exactly the transitive tolerances on $\mathbf L$.

G.~Cz\'edli (\cite{Cz}) showed that for every lattice $\mathbf L=(L,\vee,\wedge)$ and each tolerance $T$ on $\mathbf L$, the set $L/T$ of all blocks of $T$ forms a lattice again, the so-called quotient lattice of $\mathbf L$ by $T$. His famous result in this paper is that every lattice can be embedded into the quotient lattice of a distributive lattice by a suitable tolerance.

Unfortunately, quotients by tolerances cannot be introduced in a similar way even for semilattices since the join of two blocks of a tolerance on a semilattice need not exist, see e.g.\ \cite{CCH}.

Now there arises the question if results similar to those by G.~Cz\'edli can be obtained also for posets. In other words, we ask if tolerances on posets can be defined in such a way that
\begin{enumerate}[(i)]
\item in case of lattices this concept coincides with that for lattices as introduced above,
\item the set of all blocks of a tolerance forms a poset again.
\end{enumerate}
In accordance with \cite{CZ}, every block of a non-trivial tolerance should be convex and directed.

It was shown in \cite{CNZ} that every tolerance on a relatively complemented lattice is a congruence. The corresponding result for posets should hold for our new concept.

Let us note that results on tolerances on algebras can be found in the monograph \cite{Ch}, see also \cite{CCH}, \cite{CCHL} and \cite{GR}. It is worth noticing that a similar attempt for defining congruences on posets was introduced and treated in \cite{CL}.

Now let $\mathbf P=(P,\leq)$ be a poset. A {\em tolerance} on $\mathbf P$ is a reflexive and symmetric binary relation $T$ on $P$ satisfying the following conditions:
\begin{enumerate}[(1)]
\item If $(x,y),(z,u)\in T$ and $x\vee z$ and $y\vee u$ exist then $(x\vee z,y\vee u)\in T$,
\item if $(x,y),(z,u)\in T$ and $x\wedge z$ and $y\wedge u$ exist then $(x\wedge z,y\wedge u)\in T$,
\item if $x,y,z\in P$ and $(x,y),(y,z)\in T\neq P^2$ then there exist $u,v\in P$ with $u\leq x,y,z\leq v$ and $(u,y),(y,v)\in T$,
\item if $(x,y)\in T\neq P^2$ then there exists some $(z,u)\in T$ with both $z\leq x,y\leq u$ and $(v,z),(v,u)\in T$ for all $v\in P$ with $(v,x),(v,y)\in T$.
\end{enumerate}
Conditions (3) and (4) are quite natural since they are satisfied by every tolerance on a lattice. In condition (3) one can take $u:=x\wedge y\wedge z$ and $v:=x\vee y\vee z$, and in condition (4) one can take $z:=x\wedge y$ and $u:=x\vee y$.

Let $\Tol\mathbf P$ denote the set of all tolerances on $\mathbf P$. Obviously, $\bigcup\limits_{x\in P}\{x\}^2$ is the smallest tolerance on $\mathbf P$ and $P^2$ the greatest one. These {\em tolerances} are called the {\em trivial} ones. A {\em block} of a tolerance $T$ on $\mathbf P$ is a maximal subset $B$ of $P$ satisfying $B^2\subseteq T$. Let $P/T$ denote the set of all blocks of $T$. Clearly, $T=\bigcup\limits_{B\in P/T}B^2$. A {\em congruence} on $\mathbf P$ is a transitive tolerance on $\mathbf P$. Let $\Con\mathbf P$ denote the set of all congruences on $\mathbf P$. Obviously, $\bigcup\limits_{x\in P}\{x\}^2$ is the smallest congruence on $\mathbf P$ and $P^2$ the greatest one.

Let $A\subseteq P$. Then $A$ is called {\em directed} if for every $x,y\in A$ there exist $z,u\in A$ with $z\leq x,y\leq u$. Further, $A$ is called {\em convex} if for all $x,y\in A$ with $x\leq y$ we have $[x,y]\subseteq A$. If $\mathbf P$ has a bottom element $a$ and a top element $b$ then $\mathbf P$ is called {\em complemented} if for every $x\in P$ there exists some $y\in P$ satisfying $x\vee y=b$ and $x\wedge y=a$. Every such element $y$ is called a {\em complement} of $x$. The poset $\mathbf P$ is called {\em relatively complemented} if for all $x,y\in P$ with $x\leq y$ the poset $([x,y],\leq)$ is complemented. 

At first, we investigate intervals in blocks of tolerances.

\begin{proposition}\label{prop1}
Let $\mathbf P=(P,\leq)$ be a poset, $T\in\Tol\mathbf P$ and $a,b\in P$ with $a\leq b$. Then the following assertions hold:
\begin{enumerate}[{\rm(i)}]
\item If $(a,b)\in T$ then $[a,b]^2\subseteq T$,
\item if $B\in P/T$ has bottom element $a$ and top element $b$ then $B=[a,b]$.
\end{enumerate}
\end{proposition}

\begin{proof}
The case $T=P^2$ is trivial. Hence assume $T\neq P^2$.
\begin{enumerate}[(i)]
\item Assume $(a,b)\in T$ and let $c,d\in[a,b]$. Now
\begin{align*}
(a,b),(d,d)\in T & \text{ implies }(a,d)=(a\wedge d,b\wedge d)\in T, \\
(b,a),(c,c)\in T & \text{ implies }(c,a)=(b\wedge c,a\wedge c)\in T.
\end{align*}
Hence $(a,d),(c,a)\in T$ which implies $(c,d)=(a\vee c,d\vee a)\in T$ showing $[a,b]^2\subseteq T$.
\item Assume $B\in P/T$ to have bottom element $a$ and top element $b$. Then $B\subseteq[a,b]$. If $B\neq[a,b]$ then $B$ would not be a maximal subset $C$ of $P$ satisfying $C^2\subseteq T$. Therefore $B=[a,b]$ and hence $B$ is convex.
\end{enumerate}
\end{proof}

Using the previous result, we can show that blocks of non-trivial tolerances on posets share the essential properties of blocks of tolerances on lattices.

\begin{theorem}\label{th1}
Every block of a non-trivial tolerance on a poset is directed and convex.
\end{theorem}

\begin{proof}
Let $\mathbf P=(P,\leq)$ be a poset, $T$ a non-trivial tolerance on $\mathbf P$ and $B\in P/T$. We first prove that $B$ is directed. Let $a,b\in B$. Then $(a,b)\in T$. According to (4) there exists some $(c,d)\in T$ with both $c\leq a,b\leq d$ and $(x,c),(x,d)\in T$ for all $x\in P$ with $(x,a),(x,b)\in T$. Now let $e\in B$. Since $(e,a),(e,b)\in T$ we have $(e,c),(e,d)\in T$. This shows $(B\cup\{c\})^2\cup(B\cup\{d\})^2\subseteq T$. If $c\notin B$ then $B$ would not be a maximal subset $C$ of $P$ satisfying $C^2\subseteq T$. Hence $c\in B$. Analogously, we obtain $d\in B$. This shows that $B$ is directed. Now let $f,g\in B$ with $f\leq g$ and $h\in[f,g]$. Further let $i\in B$. Since $B$ is directed there exist $j,k\in B$ with $j\leq f,i$ and $g,i\leq k$. Now we have $j\leq f\leq h\leq g\leq k$, $j\leq i\leq k$ and $(j,k)\in T$. According to Proposition~\ref{prop1} (i), $(h,i)\in[j,k]^2\subseteq T$. This means that $(h,x)\in T$ for all $x\in B$ and hence $(B\cup\{h\})^2\subseteq T$. If $h\notin B$ then $B$ would not be a maximal subset $C$ of $P$ satisfying $C^2\subseteq T$. Hence $h\in B$ showing the convexity of $B$.
\end{proof}

A poset $\mathbf P=(P,\leq)$ is said to satisfy the {\em Ascending Chain Condition {\rm(ACC)} if there do not exist infinite ascending chains in $\mathbf P$, and it is said to satisfy the {\em Descending Chain Condition {\rm(DCC)} if there do not exist infinite descending chains in $\mathbf P$. In particular, every finite poset satisfies both conditions (ACC) and (DCC).

\begin{corollary}
Every block of a non-trivial tolerance on a poset satisfying the {\rm(ACC)} and the {\rm(DCC)}, especially every block of a non-trivial tolerance on a finite poset is an interval of the form $[a,b]$ with $a\leq b$.
\end{corollary}

The following two examples show that the intersection of two tolerances need not be a tolerance and that the posets $(\Tol\mathbf P,\subseteq)$ and \\
$(\Con\mathbf P,\subseteq)$ need not be lattices.

\begin{example}
Consider the poset $\mathbf P$ depicted in Figure~1:

\vspace*{-2mm}

\begin{center}
\setlength{\unitlength}{7mm}
\begin{picture}(4,6)
\put(2,1){\circle*{.3}}
\put(1,2){\circle*{.3}}
\put(3,2){\circle*{.3}}
\put(1,4){\circle*{.3}}
\put(3,4){\circle*{.3}}
\put(2,5){\circle*{.3}}
\put(2,1){\line(-1,1)1}
\put(2,1){\line(1,1)1}
\put(1,2){\line(0,1)2}
\put(1,2){\line(1,1)2}
\put(3,2){\line(-1,1)2}
\put(3,2){\line(0,1)2}
\put(2,5){\line(-1,-1)1}
\put(2,5){\line(1,-1)1}
\put(1.85,.25){$0$}
\put(.3,1.85){$a$}
\put(3.4,1.85){$b$}
\put(.3,3.85){$c$}
\put(3.4,3.85){$d$}
\put(1.85,5.4){$1$}
\put(1.2,-.75){{\rm Fig.~1}}
\end{picture}
\end{center}

\vspace*{2mm}

If
\begin{align*}
T_1 & :=\{0,a,b,c\}^2\cup\{b,c,d,1\}^2, \\
T_2 & :=\{0,a,b,d\}^2\cup\{a,c,d,1\}^2
\end{align*}
then $T_1,T_2\in\Tol\mathbf P$ are not congruences on $\mathbf P$ and
\[
T_1\cap T_2=\{0,a,b\}^2\cup\{a,c\}^2\cup\{b,d\}^2\cup\{c,d,1\}^2,
\]
but $T_1\cap T_2\notin\Tol\mathbf P$ since $\{0,a,b\}$ and $\{c,d,1\}$ are not directed.
\end{example}

\begin{example}
Consider the poset $\mathbf P$ visualized in Figure~2:
\begin{center}
\setlength{\unitlength}{7mm}
\begin{picture}(4,4)
\put(2,1){\circle*{.3}}
\put(1,2){\circle*{.3}}
\put(3,2){\circle*{.3}}
\put(1,4){\circle*{.3}}
\put(3,4){\circle*{.3}}
\put(2,1){\line(-1,1)1}
\put(2,1){\line(1,1)1}
\put(1,2){\line(0,1)2}
\put(1,2){\line(1,1)2}
\put(3,2){\line(-1,1)2}
\put(3,2){\line(0,1)2}
\put(1.85,.25){$0$}
\put(.3,1.85){$a$}
\put(3.4,1.85){$b$}
\put(.3,3.85){$c$}
\put(3.4,3.85){$d$}
\put(1.2,-.75){{\rm Fig.~2}}
\end{picture}
\end{center}

\vspace*{2mm}

If
\begin{align*}
T_1 & :=\{0,a\}^2\cup\{b\}^2\cup\{c\}^2\cup\{d\}^2, \\
T_2 & :=\{0,b\}^2\cup\{a\}^2\cup\{c\}^2\cup\{d\}^2, \\
T_3 & :=\{0,a,b,c\}^2\cup\{d\}^2, \\
T_4 & :=\{0,a,b,d\}^2\cup\{c\}^2
\end{align*}
then $T_1,T_2,T_3,T_4\in\Tol\mathbf P$ and $T_3$ and $T_4$ are minimal upper bounds of $\{T_1,T_2\}$ in $(\Tol\mathbf P,\subseteq)$ and hence $T_1\vee T_2$ does not exist, i.e.\ $(\Tol\mathbf P,\subseteq)$ is not a lattice. Observe that $T_1,\ldots,T_4\in\Con\mathbf P$. Hence, also $(\Con\mathbf P,\subseteq)$ need not be a lattice.
\end{example}

As shown by Theorem~\ref{th1}, non-trivial blocks of tolerances on posets have similar properties as those on lattices. Hence our next task is to show when the set $P/T$ of all blocks of a tolerance $T$ on a poset $\mathbf P$ is again a poset. In other words, we ask if there can be introduced a partial order relation on the set $P/T$ which in the case that $\mathbf P$ is a lattice coincides with the partial order relation induced by the lattice $(P/T,\vee,\wedge)$ introduced by G.~Cz\'edli (\cite{Cz}). For this, we need the following.

Let $\mathbf P=(P,\leq)$ be a poset, $T\in\Tol\mathbf P$ and $B_1,B_2\in P/T$. We define
\begin{align*}
B_1\sqsubseteq B_2 & \text{ if for every }b_1\in B_1\text{ there exists some }b_2'\in B_2\text{ with }b_1\leq b_2'\text{ and} \\
                   & \hspace*{5mm} \text{for every }b_2\in B_2\text{ there exists some }b_1'\in B_1\text{ with }b_1'\leq b_2.
\end{align*}

In the proof of Theorem~\ref{th3} we will use the following famous theorem by G.~Cz\'edli showing that for any lattice $\mathbf L$ and any $T\in\Tol\mathbf L$ the set $L/T$ forms again a lattice in some natural way.

\begin{theorem}\label{th2}
{\rm(}cf.\ {\rm\cite{Cz})} Let $\mathbf L=(L,\vee,\wedge)$ be a lattice, $T\in\Tol\mathbf L$ and $B_1,B_2\in L/T$. Then there exist unique $B_3,B_4\in L/T$ such that $b_1\vee b_2\in B_3$ and $b_1\wedge b_2\in B_4$ for all $b_1\in B_1$ and all $b_2\in B_2$. Put $B_1\vee B_2:=B_3$ and $B_1\wedge B_2:=B_4$. Then $(L/T,\vee,\wedge)$ is again a lattice.
\end{theorem}

Now we show that the ordering of blocks introduced above for arbitrary posets extends the lattice ordering mentioned in Theorem~\ref{th2} from lattices to posets. In the proofs of the following two theorems we frequently use Theorem~\ref{th1}.

\begin{theorem}\label{th3}
Let $\mathbf L=(L,\vee,\wedge)$ be a lattice and $T\in\Tol\mathbf L$. Then the relation $\sqsubseteq$ defined above is the partial order relation induced by the lattice $(L/T,\vee,\wedge)$.
\end{theorem}

\begin{proof}
Let $B_1,B_2\in L/T$ and let $\leq$ denote the partial order relation induced by the lattice $(L/T,\vee,\wedge)$. Then the following are equivalent:
\begin{enumerate}[(i)]
\item $B_1\leq B_2$,
\item $B_1\vee B_2=B_2$ and $B_1\wedge B_2=B_1$,
\item $b_1\vee b_2\in B_2$ and $b_1\wedge b_2\in B_1$ for all $b_1\in B_1$ and all $b_2\in B_2$,
\item for every $b_1\in B_1$ there exists some $b_2'\in B_2$ with $b_1\leq b_2'$, and for every $b_2\in B_2$ there exists some $b_1'\in B_1$ with $b_1'\leq b_2$,
\item $B_1\sqsubseteq B_2$.
\end{enumerate}
(i) $\Leftrightarrow$ (ii): \\
This is clear. \\
(ii) $\Leftrightarrow$ (iii): \\
This follows from Theorem~\ref{th2}. \\
(iii) $\Rightarrow$ (iv): \\
Let $b_1\in B_1$ and $b_2\in B_2$. Then $b_2,b_1\vee b_2\in B_2$. Since $B_2$ is directed there exists some $b_2'\in B_2$ with $b_2,b_1\vee b_2\leq b_2'$. Now $b_1\leq b_1\vee b_2\leq b_2'$. Analogously, we have $b_1,b_1\wedge b_2\in B_1$. Since $B_1$ is directed there exists some $b_1'\in B_1$ with $b_1'\leq b_1,b_1\wedge b_2$. Now $b_1'\leq b_2$. \\
(iv) $\Rightarrow$ (iii): \\
Let $b_1\in B_1$ and $b_2\in B_2$. Then there exists some $b_2'\in B_2$ with $b_1\leq b_2'$ and there exists some $b_1'\in B_!$ with $b_1'\leq b_2$. Since $B_2$ is directed there exists some $b_2''\in B_2$ with $b_2,b_2'\leq b_2''$. Now $b_2\leq b_1\vee b_2\leq b_2''$. Since $B_2$ is convex we obtain $b_1\vee b_2\in B_2$. Analogously, since $B_1$ is directed there exists some $b_1''\in B_1$ with $b_1''\leq b_1,b_1'$. Now $b_1''\leq b_1\wedge b_2\leq b_1$. Since $B_1$ is convex we obtain $b_1\wedge b_2\in B_1$. \\
(iv) $\Leftrightarrow$ (v):
This follows from the definition of $\sqsubseteq$.
\end{proof}

Now we can state and prove our result revealing the structure of $P/T$ for $T\in\Tol\mathbf P$.

\begin{theorem}
Let $\mathbf P=(P,\leq)$ be a poset and $T\in\Tol\mathbf P$. Then $(P/T,\sqsubseteq)$ is again a poset, called the {\em quotient poset $\mathbf P/T$ of $\mathbf P$ with respect to $T$}.
\end{theorem}

\begin{proof}
If $T=\bigcup\limits_{x\in P}\{x\}^2$ and $a,b\in P$ then $\{a\}\sqsubseteq\{b\}$ if and only if $a\leq b$. Hence $(P/T,\sqsubseteq)$ is a poset in this case. If $T=P^2$ then $(P/T,\sqsubseteq)$ is the one-element poset. Hence assume $T$ to be non-trivial. We are going to show that $\sqsubseteq$ is a partial order relation on $P/T$. For this purpose let $B_1,B_2,B_3\in P/T$. \\ 
First assume $b_1,b_1'\in B_1$. Since $B_1$ is directed there exist $b_1'',b_1'''\in B_1$ with $b_1''\leq b_1,b_1'\leq b_1'''$. This shows $B_1\sqsubseteq B_1$, and $\sqsubseteq$ is reflexive. \\
Now assume $B_1\sqsubseteq B_2\sqsubseteq B_1$. Let $b_1\in B_1$ and $b_2\in B_2$. Then there exist $b_1',b_1''\in B_1$ and $b_2',b_2''\in B_2$ with $b_1\leq b_2'$, $b_1'\leq b_2$, $b_2\leq b_1''$ and $b_2''\leq b_1$. Together we obtain $b_2''\leq b_1\leq b_2'$ and $b_1'\leq b_2\leq b_1''$ which because of the convexity of $B_2$ and $B_1$ yields $b_1\in B_2$ and $b_2\in B_1$. This shows $B_1\subseteq B_2$ and $B_2\subseteq B_1$, i.e.\ $B_1=B_2$, and $\sqsubseteq$ is antisymmetric. \\
Finally, assume $B_1\sqsubseteq B_2\sqsubseteq B_3$. Let $b_1\in B_1$ and $b_3)\in B_3$. Since $B_1\sqsubseteq B_2$ there exists some $b_2\in B_2$ with $b_1\leq b_2$. Since $B_2\sqsubseteq B_3$ there exists some $b_3'\in B_3$ with $b_2\leq b_3'$. Together we obtain $b_3'\in B_3$ and $b_1\leq b_2\leq b_3'$. Since $b_3\in B_3$ and $B_2\sqsubseteq B_3$ there exists some $b_2'\in B_2$ with $b_2'\leq b_3$. Since $B_1\sqsubseteq B_2$ there exists some $b_1'\in B_1$ with $b_1'\leq b_2'$. Together we obtain $b_1'\in B_1$ and $b_1'\leq b_2'\leq b_3$, i.e.\ $B_1\sqsubseteq B_3$, and $\sqsubseteq$ is transitive.
\end{proof}

\begin{example}
Consider the poset $\mathbf P$ depicted in Figure~1 and the following tolerances on $\mathbf P$:
\begin{align*}
T_1 & =\{0,a,b,c\}^2\cup\{b,c,d,1\}^2=B_1^2\cup B_2^2, \\
T_2 & =\{0,a\}^2\cup\{b,c\}^2\cup\{d,1\}^2=C_1^2\cup C_2^2\cup C_3^2.
\end{align*}
Then the quotient posets $\mathbf P/T_i$ {\rm(}$i=1,2${\rm)} are visualized in Figure~3:

\vspace*{-2mm}

\begin{center}
\setlength{\unitlength}{7mm}
\begin{picture}(0,4)
\put(0,1){\circle*{.3}}
\put(0,3){\circle*{.3}}
\put(0,1){\line(0,1)2}
\put(-.3,.25){$B_1$}
\put(-.3,3.4){$B_2$}
\put(-.65,-.75){$\mathbf P/T_1$}
\end{picture}
\quad\quad\quad\quad\quad\quad\quad\quad
\begin{picture}(0,6)
\put(0,1){\circle*{.3}}
\put(0,3){\circle*{.3}}
\put(0,5){\circle*{.3}}
\put(0,1){\line(0,1)4}
\put(-.3,.25){$C_1$}
\put(-.3,5.4){$C_3$}
\put(.3,2.85){$C_2$}
\put(-.6,-.75){$\mathbf P/T_2$}
\put(-3.2,-1.75){{\rm Fig.~3}}
\end{picture}
\end{center}

\vspace*{8mm}

\end{example}

\begin{remark}
If $B_1$ and $B_2$ are intervals of the form $[a,b]$ and $[c,d]$, respectively, then $[a,b]\sqsubseteq[c,d]$ if and only if $a\leq c$ and $b\leq d$.
\end{remark}

It was shown in \cite{CNZ} that every tolerance on a relatively complemented lattice $\mathbf L$ is a congruence. We are going to prove an analogous result also for posets.

\begin{theorem}
Let $\mathbf P=(P,\leq)$ be a relatively complemented poset. Then $\Tol\mathbf P=\Con\mathbf P$.
\end{theorem}

\begin{proof}
Let $a,b,c\in P$ and $T\in\Tol\mathbf P$ and assume $(a,b),(b,c)\in T$. If $T=P^2$ then $T\in\Con\mathbf P$. Hence assume $T\neq P^2$. According to (3) there exist $d,e\in P$ with $d\leq a,b,c\leq e$ and $(d,b),(b,e)\in T$. Since $\mathbf P$ is relatively complemented, there exists some complement $f$ of $b$ in $[d,e]$. Now $d\vee f$, $b\vee f$, $f\wedge b$ and $e\wedge e$ exist, thus
\begin{align*}
(d,b),(f,f)\in T & \text{ implies }(f,e)=(d\vee f,b\vee f)\in T, \\
(f,e),(b,e)\in T & \text{ implies }(d,e)=(f\wedge b,e\wedge e)\in T.
\end{align*}
Since $(d,e)\in T$, from Proposition~\ref{prop1} (i) we obtain $(a,c)\in[d,e]^2\subseteq T$, i.e.\ $T$ is transitive and hence $T\in\Con\mathbf P$.
\end{proof}

\begin{example}
The poset depicted in Figure~4 is relatively complemented, but not a lattice:

\vspace*{-2mm}

\begin{center}
\setlength{\unitlength}{7mm}
\begin{picture}(8,8)
\put(4,1){\circle*{.3}}
\put(1,3){\circle*{.3}}
\put(3,3){\circle*{.3}}
\put(5,3){\circle*{.3}}
\put(7,3){\circle*{.3}}
\put(1,5){\circle*{.3}}
\put(3,5){\circle*{.3}}
\put(5,5){\circle*{.3}}
\put(7,5){\circle*{.3}}
\put(4,7){\circle*{.3}}
\put(4,1){\line(-3,2)3}
\put(4,1){\line(-1,2)1}
\put(4,1){\line(1,2)1}
\put(4,1){\line(3,2)3}
\put(1,3){\line(0,1)2}
\put(1,3){\line(1,1)2}
\put(1,3){\line(2,1)4}
\put(3,3){\line(-1,1)2}
\put(3,3){\line(0,1)2}
\put(3,3){\line(2,1)4}
\put(5,3){\line(-2,1)4}
\put(5,3){\line(0,1)2}
\put(5,3){\line(1,1)2}
\put(7,3){\line(-2,1)4}
\put(7,3){\line(-1,1)2}
\put(7,3){\line(0,1)2}
\put(4,7){\line(-3,-2)3}
\put(4,7){\line(-1,-2)1}
\put(4,7){\line(1,-2)1}
\put(4,7){\line(3,-2)3}
\put(3.85,.25){$0$}
\put(.3,2.85){$a$}
\put(2.3,2.85){$b$}
\put(5.4,2.85){$c$}
\put(7.4,2.85){$d$}
\put(.3,4.85){$d'$}
\put(2.3,4.85){$c'$}
\put(5.4,4.85){$b'$}
\put(7.4,4.85){$a'$}
\put(3.85,7.4){$1$}
\put(3.2,-.75){{\rm Fig.~4}}
\end{picture}
\end{center}

\vspace*{2mm}

\end{example}

\begin{example}
Another example of a relatively complemented poset which is not directed is visualized in Figure~5:

\vspace*{-2mm}

\begin{center}
\setlength{\unitlength}{7mm}
\begin{picture}(2,4)
\put(0,1){\circle*{.3}}
\put(2,1){\circle*{.3}}
\put(0,3){\circle*{.3}}
\put(2,3){\circle*{.3}}
\put(0,1){\line(0,1)2}
\put(0,1){\line(1,1)2}
\put(2,1){\line(-1,1)2}
\put(2,1){\line(0,1)2}
\put(-.15,.25){$a$}
\put(1.85,.25){$b$}
\put(-.15,3.4){$c$}
\put(1.85,3.4){$d$}
\put(.2,-.75){{\rm Fig.~5}}
\end{picture}
\end{center}

\vspace*{2mm}

The list of congruences of $\mathbf P$ is as follows:
\begin{align*}
C_1 & =\{a\}^2\cup\{b\}^2\cup\{c\}^2\cup\{d\}^2, \\
C_2 & =\{a\}^2\cup\{c\}^2\cup\{b,d\}^2, \\
C_3 & =\{a\}^2\cup\{d\}^2\cup\{b,c\}^2, \\
C_4 & =\{b\}^2\cup\{c\}^2\cup\{a,d\}^2, \\
C_5 & =\{b\}^2\cup\{d\}^2\cup\{a,c\}^2, \\
C_6 & =\{a,c\}^2\cup\{b,d\}^2, \\
C_7 & =\{a,d\}^2\cup\{b,c\}^2,\\
C_8 & =\{a,b,c,d\}^2.
\end{align*}
The poset $(\Con\mathbf P,\subseteq)$ is depicted in Figure~6:

\vspace*{-2mm}

\begin{center}
\setlength{\unitlength}{7mm}
\begin{picture}(8,7)
\put(4,1){\circle*{.3}}
\put(1,3){\circle*{.3}}
\put(3,3){\circle*{.3}}
\put(5,3){\circle*{.3}}
\put(7,3){\circle*{.3}}
\put(3,5){\circle*{.3}}
\put(5,5){\circle*{.3}}
\put(4,6){\circle*{.3}}
\put(4,1){\line(-3,2)3}
\put(4,1){\line(-1,2)1}
\put(4,1){\line(1,2)1}
\put(4,1){\line(3,2)3}
\put(3,5){\line(0,-1)2}
\put(5,5){\line(0,-1)2}
\put(4,6){\line(-1,-1)3}
\put(4,6){\line(1,-1)3}
\put(3.7,.25){$C_1$}
\put(.15,2.85){$C_2$}
\put(2.15,2.85){$C_5$}
\put(5.25,2.85){$C_3$}
\put(7.25,2.85){$C_4$}
\put(2.15,4.85){$C_6$}
\put(5.25,4.85){$C_7$}
\put(3.7,6.4){$C_8$}
\put(3.2,-.75){{\rm Fig.~6}}
\end{picture}
\end{center}

\vspace*{2mm}

It is easy to see that $\Con\mathbf P=\Tol\mathbf P$.
\end{example}

J.~Grygiel and S.~Radelecki (\cite{GR}) showed that a partial order relation $\leq$ on the set $\Tol\mathbf L$ of tolerances on a lattice $\mathbf L$ can be introduced in such a way that for $S,T\in\Tol\mathbf L$ with $S\leq T$ a tolerance $T/S$ on the quotient lattice $\mathbf L/S$ can be defined such that the {\em Isomorphism Theorem for tolerances}
\[
(\mathbf L/S)/(T/S)\cong\mathbf L/T
\]
holds. For posets, we proceed as follows.

Let $\mathbf P=(P,\leq)$ be a poset and $S,T\in\Tol\mathbf P$. We say that $S\leq T$ if the following conditions hold:
\begin{itemize}
\item If $B_1\in P/S$ then there exists exactly one $B_2\in P/T$ with $B_1\subseteq B_2$,
\item every block of $T$ is a union of blocks of $S$.
\end{itemize}
Note that the first condition implies $S\subseteq T$. It is easy to see that $\leq$ is reflexive and antisymmetric.

In case $S\leq T$ for $S,T\in\Tol\mathbf P$ define a binary relation $T/S$ on $P/S$ as follows: For $B_1,B_2\in P/S$ we have $(B_1,B_2)\in T/S$ if there exists some $B_3\in P/T$ with $B_1,B_2\subseteq B_3$.

The following lemma is obvious.

\begin{lemma}\label{lem1}
Let $\mathbf P=(P,\leq)$ be a poset and $S,T\in\Tol\mathbf P$ and assume $S\leq T$. Then $T/S$ is reflexive and symmetric.
\end{lemma}

\begin{proof}
If $B_1\in P/S$ then there exists some $B_2\in P/T$ with $B_1\subseteq B_2$ which shows $(B_1,B_1)\in T/S$. The symmetry of $T/S$ is clear.
\end{proof}

Example~\ref{ex2} shows that the Isomorphism Theorem as it was mentioned above in the case of lattices does not hold for tolerances on posets. However, for tolerances $S$ and $T$ on a poset $\mathbf P=(P,\leq)$ such that $S\leq T$ and $T/S\in\Tol(\mathbf P/S)$ we can construct an injective mapping from $P/T$ to $(P/S)/(T/S)$, see the following theorem.

\begin{theorem}\label{th5}
Let $\mathbf P=(P,\leq)$ be a poset and $S,T\in\Tol\mathbf P$ and assume $S\leq T$. Further assume that $T/S$ satisfies {\rm(1)} -- {\rm(4)}. Then
\begin{enumerate}[{\rm(i)}]
\item $T/S\in\Tol(\mathbf P/S)$,
\item $|(P/S)/(T/S)|\geq|P/T|$.
\end{enumerate}
\end{theorem}

\begin{proof}
\
\begin{enumerate}[(i)]
\item follows from Lemma~\ref{lem1}.
\item Define $f(B_1):=\{B_2\in P/S\mid B_2\subseteq B_1\}$ for all $B_1\in P/T$. Observe that because of $S\leq T$, $f(B_1)\neq\emptyset$ for all $B_1\in P/T$. We show that $f$ is an injective mapping from $P/T$ to $(P/S)/(T/S)$. Let $B_1\in P/T$. Then $(B_2,B_3)\in T/S$ for all $B_2,B_3\in f(B_1)$. Now let $B_4\in P/S$ and assume $(B_4,B_2)\in T/S$ for all $B_2\in f(B_1)$. Let $B_5\in f(B_1)$. Then there exists some $B_6\in P/T$ with $B_4,B_5\subseteq B_6$. Now $B_5\in P/S$, $B_1,B_6\in P/T$ and $B_5\subseteq B_1,B_6$. Since $S\leq T$ we conclude $B_1=B_6$. This shows $B_4\subseteq B_6=B_1$, i.e.\ $B_4\in f(B_1)$. Therefore $f(B_1)\in(P/S)/(T/S)$. Injectivity of $f$ follows from the fact that because of $S\leq T$ we have $B_1=\bigcup\limits_{B_2\in f(B_1)}B_2$ for all $B_1\in P/T$.
\end{enumerate}
\end{proof}

The following example demonstrates that in some cases not only $|(P/S)/(T/S)|\geq|P/T|$, but $|(P/S)/(T/S)|=|P/T|$.

\begin{example}\label{ex1}
Let $\mathbf P$ to be the poset visualized in Figure~2 and put
\begin{align*}
S & :=\{0,a\}^2\cup\{b,c\}^2\cup\{d\}^2=B_1^2\cup B_2^2\cup B_3^2, \\
T & :=\{0,a,b,c\}^2\cup\{d\}^2=C_1^2\cup C_2^2.
\end{align*}
Then $S,T\in\Tol\mathbf P$, $S\leq T$ and $\mathbf P/S$ and $\mathbf P/T$ look as follows:

\vspace*{-2mm}

\begin{center}
\setlength{\unitlength}{7mm}
\begin{picture}(2,3)
\put(1,1){\circle*{.3}}
\put(0,2){\circle*{.3}}
\put(2,2){\circle*{.3}}
\put(1,1){\line(-1,1)1}
\put(1,1){\line(1,1)1}
\put(.65,.25){$B_1$}
\put(-.35,2.4){$B_2$}
\put(1.65,2.4){$B_3$}
\put(.35,-.75){$\mathbf P/S$}
\put(3.4,-1.75){{\rm Fig.~7}}
\end{picture}
\quad\quad\quad\quad\quad\quad\quad\quad
\begin{picture}(2,1)
\put(0,1){\circle*{.3}}
\put(2,1){\circle*{.3}}
\put(-.35,.25){$C_1$}
\put(1.65,.25){$C_2$}
\put(.4,-.75){$\mathbf P/T$}
\end{picture}
\end{center}

\vspace*{7mm}

Further,
\[
T/S=\{B_1,B_2\}^2\cup\{B_3\}^2=D_1^2\cup D_2^2\in\Tol(\mathbf P/S),
\]
and $(\mathbf P/S)/(T/S)$ looks as follows:

\vspace*{-3mm}

\begin{center}
\setlength{\unitlength}{7mm}
\begin{picture}(2,1)
\put(0,1){\circle*{.3}}
\put(2,1){\circle*{.3}}
\put(-.35,.25){$D_1$}
\put(1.65,.25){$D_2$}
\put(-.75,-.75){$(\mathbf P/S)/(T/S)$}
\put(.25,-1.75){{\rm Fig.~8}}
\end{picture}
\end{center}

\vspace*{9mm}

This shows $\mathbf P/T\cong(\mathbf P/S)/(T/S)$.
\end{example}

If in addition to the assumptions of Theorem~\ref{th5} we assume $S,T\in\Con\mathbf P$ then we can show that there exists even a bijection from $(P/S)/(T/S)$ to $P/T$ which is, moreover, order-preserving. Unfortunately, in general this bijection is not an order-isomorphism since its inverse need not be order-preserving, see Example~\ref{ex2}.

\begin{theorem}\label{th4}
Let $\mathbf P=(P,\leq)$ be a poset and $S,T\in\Con\mathbf P$ and assume $S\leq T$ and $T/S\in\Tol(\mathbf P/S)$. Then there exists a bijective order-preserving mapping from \\
$(\mathbf P/S)/(T/S)$ onto $\mathbf P/T$.
\end{theorem}

\begin{proof}
Since $S,T\in\Con\mathbf P$, the blocks of $S$ as well as the blocks of $T$ form a partition of $P$ and every block of $T$ is a disjoint union of blocks of $S$. Hence two blocks of $S$ are in relation $T/S$ if and only if they are contained in the same block of $T$. It is clear that the blocks of $T/S$ are exactly the sets of the form $\{C\in P/S\mid C\subseteq B\}$ with $B\in P/T$. Define a mapping $f$ from $(P/S)/(T/S)$ to $P/T$ by $f(B):=\bigcup\limits_{C\in B}C$ for all $B\in(P/S)/(T/S)$. Since every block of $T/S$ is the set of all blocks of $S$ that are contained in a fixed block of $T$, $f$ is a bijection. We show that it is order-preserving. Let $B_1,B_2\in(P/S)/(T/S)$ and assume $B_1\sqsubseteq B_2$. Let $a\in f(B_1)$. Then there exists some $B_3\in B_1$ with $a\in B_3$. Since $B_1\sqsubseteq B_2$ there exists some $B_4\in B_2$ with $B_3\sqsubseteq B_4$. Hence there exists some $b\in B_4$ with $a\leq b$. Therefore $b\in f(B_2)$ and $a\leq b$. Conversely, let $b\in f(B_2)$. Then there exists some $B_4\in B_2$ with $b\in B_4$. Since $B_1\sqsubseteq B_2$ there exists some $B_3\in B_1$ with $B_3\sqsubseteq B_4$. Hence there exists some $a\in B_3$ with $a\leq b$. Therefore $a\in f(B_1)$ and $a\leq b$. Together, we have proved that $f(B_1)\sqsubseteq f(B_2)$, i.e.\ $f$ is order-preserving.
\end{proof}

If we apply the proof of Theorem~\ref{th4} to Example~\ref{ex1} we obtain $f(D_i)=C_i$ for $i=1,2$.

That the inverse of the bijective order-preserving mapping mentioned in Theorem~\ref{th4} need not be order-preserving is demonstrated by the following example.

\begin{example}\label{ex2}
Let $\mathbf P$ be the poset depicted in Figure~1 and put
\begin{align*}
S & :=\{0,a\}^2\cup\{b\}^2\cup\{c\}^2\cup\{d,1\}^2=B_1^2\cup B_2^2\cup B_3^2\cup B_4^2, \\
T & :=\{0,a\}^2\cup\{b,c\}^2\cup\{d,1\}^2=C_1^2\cup C_2^2\cup C_3^2.
\end{align*}
Then $S,T\in\Tol\mathbf P$, $S\leq T$ and $\mathbf P/S$ and $\mathbf P/T$ look as follows:

\vspace*{-2mm}

\begin{center}
\setlength{\unitlength}{7mm}
\begin{picture}(2,4)
\put(0,1){\circle*{.3}}
\put(2,1){\circle*{.3}}
\put(0,3){\circle*{.3}}
\put(2,3){\circle*{.3}}
\put(0,1){\line(0,1)2}
\put(0,1){\line(1,1)2}
\put(2,1){\line(-1,1)2}
\put(2,1){\line(0,1)2}
\put(-.35,.25){$B_1$}
\put(1.65,.25){$B_2$}
\put(-.35,3.4){$B_3$}
\put(1.65,3.4){$B_4$}
\put(.35,-.75){$\mathbf P/S$}
\end{picture}
\quad\quad\quad\quad\quad\quad\quad\quad
\begin{picture}(0,6)
\put(0,1){\circle*{.3}}
\put(0,3){\circle*{.3}}
\put(0,5){\circle*{.3}}
\put(0,1){\line(0,1)4}
\put(-.3,.25){$C_1$}
\put(-.3,5.4){$C_3$}
\put(.3,2.85){$C_2$}
\put(-.6,-.75){$\mathbf P/T$}
\put(-3.7,-1.75){{\rm Fig.~9}}
\end{picture}
\end{center}

\vspace*{8mm}

Further,
\[
T/S=\{B_1\}^2\cup\{B_2,B_3\}^2\cup\{B_4\}^2=D_1^2\cup D_2^2\cup D_3^2\in\Tol(\mathbf P/S),
\]
and $(\mathbf P/S)/(T/S)$ looks as follows:

\vspace*{-3mm}

\begin{center}
\setlength{\unitlength}{7mm}
\begin{picture}(2,4)
\put(0,1){\circle*{.3}}
\put(0,3){\circle*{.3}}
\put(1,2){\circle*{.3}}
\put(0,1){\line(0,1)2}
\put(-.3,.25){$D_1$}
\put(-.3,3.4){$D_3$}
\put(1.25,1.75){$D_2$}
\put(-1.75,-.75){$(\mathbf P/S)/(T/S)$}
\put(-.9,-1.75){{\rm Fig.~10}}
\end{picture}
\end{center}

\vspace*{10mm}

The mapping $f$ from the proof of Theorem~\ref{th4} maps $D_i$ onto $C_i$ for $i=1,2,3$. Since $C_1\sqsubseteq C_2$, but $f^{-1}(C_1)=D_1\not\sqsubseteq D_2=f^{-1}(C_2)$, the mapping $f^{-1}$ is not order-preserving. Even more, there does not exist a bijective order-preserving mapping from $\mathbf P/T$ to $(\mathbf P/S)/(T/S)$. Hence, the Isomorphism Theorem for tolerances on posets does not hold in general even in the case when the tolerances in question are congruences.
\end{example}

Authors' addresses:

Ivan Chajda \\
Palack\'y University Olomouc \\
Faculty of Science \\
Department of Algebra and Geometry \\
17.\ listopadu 12 \\
771 46 Olomouc \\
Czech Republic \\
ivan.chajda@upol.cz

Helmut L\"anger \\
TU Wien \\
Faculty of Mathematics and Geoinformation \\
Institute of Discrete Mathematics and Geometry \\
Wiedner Hauptstra\ss e 8-10 \\
1040 Vienna \\
Austria, and \\
Palack\'y University Olomouc \\
Faculty of Science \\
Department of Algebra and Geometry \\
17.\ listopadu 12 \\
771 46 Olomouc \\
Czech Republic \\
helmut.laenger@tuwien.ac.at
\end{document}